\documentclass[a4paper]{amsart}

\usepackage[french,english]{babel}
\usepackage{amsmath}
\usepackage{amsfonts}
\usepackage{amssymb}
\usepackage{amscd}
\usepackage{yfonts}
\usepackage{mathrsfs}

\newtheorem{theorem}{Theorem}[section]
\newtheorem{proposition}[theorem]{Proposition}

\newtheorem{theoreme}{Th\'eor\`eme}[section]

\theoremstyle{definition}
\newtheorem{definition}[theorem]{Definition}

\theoremstyle{remark}
\newtheorem{remark}[theorem]{Remark}

\newtheorem{remarque}[theorem]{Remarque}

\DeclareMathOperator*{\ad}{ad}
\DeclareMathOperator*{\gal}{Gal}
\newcommand{\vectornorm}[1]{\left|\left|#1\right|\right|}

\numberwithin{equation}{section}

\newcommand{\abs}[1]{\lvert#1\rvert}

\begin{document}

\title{$m$-bigness in compatible systems}

\author{Paul-James White}

\address{Institut de Math\'ematiques de Jussieu \\ Universit\'e de Paris 7 \\ Paris \\ France}

\email{pauljames@math.jussieu.fr}

\thanks{The author gratefully acknowledges the support of the Fondation Sciences Math\'ematiques de Paris}


\date{September 13, 2010}

\begin{abstract}
\selectlanguage{english}
Taylor-Wiles type lifting theorems allow one to deduce that if $\rho$ is a ``sufficiently nice'' $l$-adic representation of the absolute Galois group of a number field
whose semi-simplified reduction modulo $l$, denoted $\overline{\rho}$, comes from an automorphic representation then so does $\rho$.  
The recent lifting theorems of Barnet-Lamb-Gee-Geraghty-Taylor impose
a technical condition, called \emph{$m$-big}, upon the residual representation $\overline{\rho}$.  
Snowden-Wiles proved that for a sufficiently  irreducible
compatible system of Galois representations, the residual images are \emph{big} at a set of places of Dirichlet density $1$.  We demonstrate the analogous result in the \emph{$m$-big} setting using a mild generalization of their argument.

\vskip 0.5\baselineskip

\selectlanguage{francais}
% Text of abstract in French
\noindent{\bf R\'esum\'e} \vskip 0.5\baselineskip \noindent
{\bf $m$-bigness dans les syst\`emes compatibles. }
Les Th\'eor\`emes de type Taylor-Wiles indiquent  qu'une repr\'esentation $l$-adique du groupe Galois d'un corps de nombre est automorphe si sa r\'eduction modulo $l$ est automorphe et si cette repr\'esentation satisfait de bonnes propri\'et\'es.  Une condition technique mais cruciale qui appara\^it dans le travail r\'ecent de Barnet-Lamb-Gee-Geraghty-Taylor est que la repr\'esentation r\'esiduelle soit \emph{$m$-big}.  Snowden-Wiles ont demontr\'e que pour un syst\`eme compatible de reprŽsentations suffisamment irr\'eductibles, que les images r\'esiduelles sont alors \emph{big} pour un ensemble de Dirichlet densit\'e 1.  Nous d\'emontrons ici un r\'esultat analogue dans le cadre de \emph{$m$-big} par une g\'en\'eralisation de la d\'emonstration de Snowden-Wiles.

\end{abstract}
\maketitle

\selectlanguage{english}

\section{Introduction}
We begin by recalling the condition \emph{$m$-big} (cf.~\cite[Definition 7.2]{BGHT}).
Let $m$ be a positive integer, let $l$ be a rational prime, let $k$ be a finite field of characteristic $l$, let $V$
be a finite dimensional $k$-vector space and let $G \subset GL(V)$ be a subgroup.
For $g \in GL(V)$ and $\alpha \in k$, we shall write 
$h_g$ for the \emph{characteristic polynomial} of $g$ and $V_{g,\alpha}$ for the 
\emph{$\alpha$-generalized eigenspace} of $g$.

\begin{definition}
\label{definition_mbig}
The subgroup $G$ is said to be \emph{$m$-big} if it satisfies the following properties.
\begin{itemize}
\item[(B1)] The group $G$ has no non-trivial quotient of $l$-power order.
\item[(B2)] The space $V$ is absolutely irreducible as a $G$-module.
\item[(B3)] $H^1(G, \ad^{\circ} V) = 0$
\item[(B4)] For all irreducible $G$-submodules $W$ of $\ad\ V$, there exists $g \in G, \alpha \in k$ and $f \in W$
such that:
\begin{itemize}
\item The composite $V_{g,\alpha} \hookrightarrow V \stackrel{f}{\rightarrow} V \twoheadrightarrow V_{g,\alpha}$ is non-zero.
\item $\alpha$ is a simple root of $h_g$.
\item If $\beta \in \overline{k}$ is another root of $h_g$ then $\alpha^m \neq \beta^m$. 
\end{itemize}
\end{itemize}
\end{definition}

\begin{remark}
The condition \emph{big} appearing in~\cite{snowden_wiles} corresponds here to the condition $\emph{$1$-big}$.
\end{remark}

Our main result is the following:
\begin{theorem}
\label{theorem_introduction_theorem}
Let $F$ be a number field, let $E$ be a Galois extension of $\mathbf{Q}$, let $L$ be a full set of places of $E$ and for each
 $w \in L$, let $\rho_w : \gal(\overline{\mathbf{Q}}/F) \rightarrow GL_n(E_w)$ be a continuous representation and let $\Delta_w \subset \gal(\overline{\mathbf{Q}}/F)$ be a normal open subgroup.  Assume that the following properties are satisfied.
 \begin{enumerate}
 	\item[-] The $\rho_w$ form a compatible system of representations. 
 	\item[-] $\rho_w$ is absolutely irreducible when restricted to any open subgroup of 
	$\gal(\overline{\mathbf{Q}}/F)$ for all $w \in L$.
	\item[-] $\gal(\overline{\mathbf{Q}}/F) / \Delta_w$ is cyclic of order prime to $l$ 
		  where $l$ denotes the residual characteristic of $w$, for all $w \in L$.
	\item[-]
	$	 [\gal(\overline{\mathbf{Q}}/F) : \Delta_w]  \rightarrow \infty$
		as $w \rightarrow \infty$.
 \end{enumerate}
 Then there exists a set of places $P$ of $\mathbf{Q}$ of Dirichlet density $1/[E:\mathbf{Q}]$, all of which split completely in $E$, such that, for all $w \in L$ lying above a place $l \in P$:
 \begin{enumerate}
 \item[i)] $\overline{\rho}_w(\Delta_w)$ is an $m$-big subgroup of  $GL_n(\mathbf{F}_l)$ .
 \item[ii)] $[ \ker \ad \overline{\rho}_w : \Delta_w \cap \ker \ad \overline{\rho}_w] > m$.
 \end{enumerate}
   \end{theorem}
     Here, as usual, $\overline{\rho}_w$ denotes the semi-simplified reduction modulo $l$ of $\rho_w$.
  For the definition of a compatible system and a full set of places, see Section~\ref{section_m_bigness_for_compatible_systems}.

\begin{remark}
The first part of the theorem is a mild generalization, from the setting of bigness to $m$-bigness, of the main result of Snowden-Wiles~\cite{snowden_wiles}.  The result shall be proved by considering their arguments in the $m$-bigness setting combined with a slight strengthening of~\cite[Proposition 4.1]{snowden_wiles} by
Proposition~\ref{proposition_stronger_version_higher_regular_elements_sw}.

The second part of the theorem proves another technical result required for the application of
automorphy lifting theorems (see~\cite{taylor_et_al_modularity_results}).  The proof of this result uses an argument of 
Barnet-Lamb-Gee-Geraghty-Taylor that originally appeared in~\cite{taylor_et_al_modularity_results}.
\end{remark}

The format of this article mirrors that of Snowden-Wiles~\cite{snowden_wiles}.  We content ourselves here to remark upon the minor changes to~\cite{snowden_wiles} that are needed to obtain the above result.

\subsection{Notation}
Our notation is as in Snowden-Wiles~\cite{snowden_wiles}.  More specifically, unless explicitly mentioned otherwise, we adhere to the following conventions.  Reductive algebraic groups are assumed connected. A semi-simple algebraic group $G$ defined over a field $k$ is \emph{simply connected} if the root datum of $G_{\overline{k}}$ is simply connected.
If $S$ is a scheme, then a group scheme $G/S$ is  \emph{semi-simple} if it is smooth, affine and its geometric fibers are  semi-simple connected algebraic groups.

\section*{Acknowledgements}
I wish to thank Michael Harris for his continual support and direction.
I would also like to thank the referee for their helpful comments.    

\section{Elementary properties of {$m$-bigness}}
\label{section_elementary_properties_m_bigness}
In~\cite[\S2]{snowden_wiles}, a series of results concerning elementary properties of bigness are demonstrated.  We remark that the arguments appearing there trivially generalize to give the following results on $m$-bigness. 

\begin{proposition}
\label{proposition_m_big_subgroup_m_big}
Let $H$ be a normal subgroup of $G$.
If $H$ satisfies the properties (B2), (B3) and (B4) then $G$ does as well.
In particular, if $H$ is $m$-big and the index $[G:H]$ is prime to $l$ then $G$ is $m$-big.
\end{proposition}

\begin{proposition}
\label{proposition_m_big_iff_scalar_m_big}
The group $G$ is $m$-big if and only if $k^\times G$ is $m$-big where $k^\times$ denotes
the group of scalar matrices in $GL(V)$.
\end{proposition}

\begin{proposition}
Let $k'/k$ be a finite extension, let $V' = V \otimes_k k'$ and let 
 $G$ be an $m$-big subgroup of $GL(V)$. Then $G$ is also an $m$-big subgroup of $GL(V')$.
\end{proposition}

\section{Highly regular elements of semi-simple groups}
\label{section_highly_regular_elements_semi_simple_groups}
We recall the norm utilized by Snowden-Wiles~\cite[\S3.2]{snowden_wiles}.
Let $k$ be a field, let $G/k$ be a reductive algebraic group and let $T_{\overline{k}}$ be a maximal torus of $G \times_k \overline{k}$.  For $\lambda \in X^*(T_{\overline{k}})$ a weight, one defines $\vectornorm{\lambda} \in \overline{k}$ to be the maximal value of $\abs{\langle \lambda, \alpha^\vee \rangle}$ as $\alpha$ runs through the roots of  $G \times_k \overline{k}$ with respect to $T_{\overline{k}}$.  For $V$ a representation of 
$G$, one defines $\vectornorm{V}$ to be the maximal value of $\vectornorm{\lambda}$ where $\lambda$ runs through the weights $\lambda$ appearing in $V \otimes_k \overline{k}$. 

The following result is a slight strengthening of~\cite[Proposition 4.1]{snowden_wiles}.

\begin{proposition}
\label{proposition_stronger_version_higher_regular_elements_sw}
Let $k$ be a finite field of cardinality $q$,
Let $G/k$ be a semi-simple algebraic group, let $T$ be a maximal torus of $G$ defined over $k$,
let $m$ and $n$ be positive integers and assume that $q$ is large compared to $\dim G$, $n$ and $m$.
Then, there exists an element $g \in T(k)$ for which the map
\[
	\{ \lambda \in X^*(T_{\overline{k}}) :
	 \vectornorm{\lambda} < n\} \rightarrow \overline{k}^\times,
	\ \ \ \ \ \lambda \mapsto \lambda(g)^m
\]
is injective.
\end{proposition}

\begin{proof}
The proof shall follow that of~\cite[Proposition 4.1]{snowden_wiles} with the difference that we are considering here characters of the form $\lambda^m$ instead of $\lambda$.

To begin let
$S := \{ \lambda \in X^*(T_{\overline{k}}) : \lambda \neq 1,\  \vectornorm{\lambda} < 2n\}$.
 We claim that
\[
T(k) \not\subset \bigcup_{\lambda \in S} \ker \lambda^m
\]
This is equivalent to the statement 
\[
T(k) \neq  \bigcup_{\lambda \in S} \ker  \lambda^m \cap T(k)
\]
The later statement shall be proved by considering the cardinality of the two sides.
Firstly, by ~\cite[Lemma 4.2]{snowden_wiles},
$ \abs{T(k)} \geq (q-1)^r$
where $r$ denotes the rank of $T$.
Consider now the right hand side.  We remark that for $\lambda \in X^*(T_{\overline{k}})$, 
\[
\abs{\ker \lambda^m \cap T(k)} \leq R_{m,q} \abs{\ker \lambda \cap T(k)}
\]
where $R_{m,q}$ denotes the cardinality of the kernel of the map 
\[
k^\times \rightarrow k^\times, \ \ \ \  k \mapsto k^m
\]
Furthermore, we can ensure that  $R_{m,q}/q$ is as small as desired simply by considering $q$ sufficiently large with respect to $m$.
We can now bound the cardinality of the right hand side by
\[
N R_{m,q} M
\]
where the terms are defined as follows.
\begin{itemize}
\item $N$ is equal to the cardinality of $S$, which by~\cite[Lemma 4.3]{snowden_wiles}
is bounded in terms of $\text{dim } G$ and $n$.
\item $M$ is equal to the maximum cardinality of $ \ker \lambda \cap T(k)$ for $\lambda \in S$,
which by~\cite[Lemma 4.4]{snowden_wiles} is bounded by $C(q+1)^{r-1}$ for some constant $C$ depending only upon
$\dim G$ and $n$.
\end{itemize}
Thus for $q$ sufficiently large with respect to $\dim G$, $n$ and $m$, the cardinality of the right hand side is strictly less than that of the cardinality of the left hand side and the claim follows.

As such we can choose a  $g \in T(k)$ such that $g \not\in \ker \lambda^m$ for all $\lambda \in S$.
Then, for all
  $\lambda, \lambda' \in X^*(T_{\overline{k}})$ such that $\lambda \neq \lambda'$, $\vectornorm{\lambda} < n$ and 
  $\vectornorm{\lambda'}< n$,
  we have that $\lambda -  \lambda' \in S$ and it follows that $\lambda(g)^m \neq \lambda'(g)^m$.
\end{proof}

\section{$m$-bigness for algebraic representations}
\label{section_m_big_algebraic_representations}
We show here that~\cite[Proposition 5.1]{snowden_wiles} naturally generalizes to the setting of $m$-bigness.

\begin{proposition}
\label{proposition_bigness_algebraic_representations}
Let $m$ be a positive integer, let $k$ be a finite field, let $G/k$ be a reductive algebraic group and let $\rho$ be an absolutely 
irreducible representation of $G$ on a $k$-vector space $V$.
Assume that the characteristic of $k$ is large compared to $m$, $\dim V$ and $\vectornorm{V}$.
Then $\rho(G(k))$ is an $m$-big subgroup of $GL(V)$.
\end{proposition}
\begin{proof}
Firstly, we note that by~\cite[Proposition 5.1]{snowden_wiles} the conditions (B1), (B2) 
and (B3) are satisfied.
Thus, it only remains to check the condition (B4) (in the $m$-bigness setting).
The proof is almost identical to the $1$-bigness case
(cf.~\cite[Proposition 4.1]{snowden_wiles}); the sole difference comes from appealing to Proposition~\ref{proposition_stronger_version_higher_regular_elements_sw} in lieu of~\cite[Proposition 4.1]{snowden_wiles}.

More specifically, 
as in~\cite[Proposition 5.1]{snowden_wiles},
one begins by reducing to the case where $G$ is semi-simple, simply connected and
the kernel of $\rho$ is finite.
Choose 
 a Borel $B$ of $G$ defined over $k$; this is possible as every reductive group scheme defined over a finite field is quasi-split.
Let $T$ be a maximal torus of $B$ and let $U$ be the unipotent radical of $B$.
% and let $U$ be the unipotent radical of $B$.
The representation  $V_{\overline{k}} = V \otimes_k \overline{k}$ decomposes via its weights:
\[
	V_{\overline{k}} = \bigoplus_{\lambda \in S} V_{\overline{k},\lambda}
\]
where $S$ denotes the set of weights of $(G_{\overline{k}}, T_{\overline{k}})$.
By Proposition~\ref{proposition_stronger_version_higher_regular_elements_sw}, we can find a $g \in T(k)$ such that
 \[
 \lambda(g)^m \neq \lambda'(g)^m
 \text{ for all distinct $\lambda, \lambda' \in S$}
 \]
 We remark that (ignoring multiplicity) the eigenvalues of $g$ in $V_{\overline{k}}$ are equal to 
 $ \{\lambda(g) : \lambda \in S\}$.  
 %By the properties of $g$,  the $m$-th powers of $\lambda(g)$ are distinct (hence the $\lambda(g)$ themselves are also distinct).  
 It follows that the generalized 
$g$-eigenspaces are equal to the weight spaces:
\[
V_{\overline{k},g,\lambda(g)} = V_{\overline{k}, \lambda} \text{ for all $ \lambda \in S$}
\]

Let $\lambda_0$ be the highest weight space (with respect to $B$) and let
 $V_{\overline{k},0} := V_{\overline{k},\lambda_0} $ be the corresponding 
highest weight space.  In fact $V_{\overline{k},0} = V^U \otimes_k \overline{k}$ and as such
$\lambda_0(g) \in k$.
By~\cite[Proposition 3.7]{snowden_wiles}, $V_{\overline{k},0}$ is
$1$-dimensional.  
That is, $\lambda_0(g)$ is a simple root of $h_g$, the characteristic polynomial of $g$.
  Furthermore, by the properties of $g$, the $m$-th powers of the roots of $h_g$ are distinct.

Finally it is shown in the proof of~\cite[Proposition 5.1]{snowden_wiles}  that for each irreducible $G$-submodule $W$ of $\ad  V$, there exists a $f \in W$ such that the composite
\[
 V_{g, \lambda_0(g)} \hookrightarrow V \stackrel{f}{\rightarrow} V \twoheadrightarrow V_{g,\lambda_0(g)}
 \]
  is non-zero.
\end{proof}

\section{$m$-bigness for nearly hyperspecial groups}
\label{section_mbig_nearly_hyperspecial_groups}
Let $l$ be a rational prime, let $K/\mathbf{Q}_l$ be a finite field extension,
let $\mathcal{O}_K$ be the ring of integers and let $k$ be the residue field.
For $G$ an algebraic group over $K$, we define the following $K$-algebraic groups.
\begin{enumerate}
\item[-] $G^\circ$ : the connected identity component.
\item[-] $G^{\text{ad}}$ : the adjoint algebraic group, which is the quotient of $G^\circ$ by its radical.
\item[-] $G^{\text{sc}}$ : the simply connected algebraic group cover of $G^{\text{ad}}$.
\end{enumerate}
We have the natural maps:
\[
	G \stackrel{\sigma}{\longrightarrow} G^{\text{ad}} \stackrel{\tau}{\longleftarrow} G^{\text{sc}}
\]
Following Snowden-Wiles~\cite{snowden_wiles}, we shall call a subgroup
$\Gamma \subset G(K)$ \emph{nearly hyperspecial} if
$\tau^{-1}\left(\sigma(\Gamma)\right)$ is a hyperspecial subgroup of $G^{\text{sc}}(K)$.

\begin{proposition}
\label{proposition_bigness_nearly_hyperspecial_groups_improvement_wiles_snowden}
Let $m$ be a positive integer,
let $\Gamma$ be a profinite group and let $\Delta \subset \Gamma$ be an open normal subgroup.
Let $\rho : \Gamma \rightarrow GL_n(K)$ be a continuous representation
and let $G$ be the Zariski closure of its image.
  Assume that the following properties are satisfied.
\begin{enumerate}
	\item[-] The characteristic $l$ of $k$ is large compared to $n$ and $m$.
	\item[-] The restriction of $\rho$ to any open subgroup of $\Gamma$ is absolutely irreducible.
	\item[-] The index of $G^\circ$ in $G$ is small compared to $l$.
	\item[-] The subgroup $\rho(\Gamma) \cap G^\circ(K)$ of $G^\circ$ is nearly hyperspecial.
	\item[-] $\Gamma/\Delta$ is cyclic of order prime to $l$.
\end{enumerate}
Then the following holds.
\begin{enumerate}
	\item[-] $\overline{\rho}(\Delta)$ is an $m$-big subgroup of $GL_n(k)$.
	\item[-] There exists a constant $C$ depending only upon $n$ such that
		\[
	[\ker \ad \overline{\rho} : \Delta \cap \ker \ad \overline{\rho}] >
		[ \Gamma : \Delta] / C	
\]
\end{enumerate}
\end{proposition}

\begin{proof}
Let us remark that
the first statement is proved in almost the same way as the proof of~\cite[Proposition 6.1]{snowden_wiles}.
There are two minor differences, firstly we appeal here to Proposition~\ref{proposition_bigness_algebraic_representations} instead of~\cite[Proposition 5.1]{snowden_wiles}
and secondly we use an argument of Barnet-Lamb-Gee-Geraghty-Taylor to deduce the $m$-bigness of $\overline{\rho}(\Delta)$ instead of $\overline{\rho}(\Gamma)$.
 The proof of the second statement  is due to Barnet-Lamb-Gee-Geraghty-Taylor
and originally appeared in~\cite{taylor_et_al_modularity_results}.

Let $\Gamma^\circ = \rho^{-1}(G^\circ)$ and let $\Delta^\circ = \Delta \cap \Gamma^\circ$.
 Then $\overline{\rho}(\Delta^\circ)$ is a normal subgroup of $\overline{\rho}(\Delta)$ and its index
 divides $[G : G^\circ] [\Gamma : \Delta]$, which, by assumption, is prime to $l$ (recall  $l$ is sufficiently large with respect to $[G : G^\circ]$).
Thus, by Proposition~\ref{proposition_m_big_subgroup_m_big}, to prove that
$\overline{\rho}(\Delta)$ is $m$-big it suffices to prove that $\overline{\rho}(\Delta^\circ)$ is $m$-big.  Similarly, to prove the second part of the theorem it clearly suffices to prove the analogous statement for $\Gamma^\circ$ and $\Delta^\circ$.  As such, we can now assume that $G=G^\circ$.

Let $V = K^n$ be the representation space of $\rho$.
By~\cite[Lemma 6.3]{snowden_wiles}, we can find the following.
\begin{itemize}
	\item A $\Gamma$-stable lattice $\Lambda$ in $V$.
	\item A semi-simple group scheme $\widetilde{G}/\mathcal{O}_K$ whose generic fiber is equal to $G^{sc}$.
	\item A representation $r : \widetilde{G} \rightarrow GL(\Lambda)$ which induces the natural map
	$G^{sc} \rightarrow G$ on the generic fiber.
\end{itemize}
These objects can be chosen such that
\begin{itemize}
	\item $\mathcal{O}_K^\times \cdot r(\widetilde{G}(\mathcal{O}_K))$ is an open normal subgroup of $\mathcal{O}_K^\times \cdot \rho(\Gamma)$, whose index  can be bounded by a constant $C$ defined in terms of $n$.
\end{itemize}
Furthermore, the generic fiber of $r$ is necessarily an absolutely irreducible representation of
$\widetilde{G}_K$ on $V$.

By~\cite[Proposition 3.5]{snowden_wiles}, $\Lambda \otimes_{\mathcal{O}_K} k$ is an absolutely irreducible representation of $\widetilde{G} \times_{\mathcal{O}_K} k$ and its norm is bounded in terms of $n$.
Now $\widetilde{G} \times_{\mathcal{O}_K} k$ is a semi-simple, simply connected, algebraic group and hence a finite product of simple, simply connected, $k$-algebraic groups.  As $l > 4$, we have that 
 $\widetilde{G}(k)$ is perfect (cf.~\cite{steinberg_lectures}). 
 It follows, as  $\Delta$ is a normal subgroup of $\Gamma$ whose quotient is abelian,  that we have the following chain of normal subgroups.
 \[
k^\times r(\widetilde{G}(k)) \leq k^\times  \overline{\rho}(\Delta) \leq k^\times \overline{\rho}(\Gamma)
 \]
 Furthermore, $[k^\times \overline{\rho}(\Gamma) : k^\times r(\widetilde{G}(k))] < C$.
  The second part of the theorem is now immediate.
 
 The first part of the theorem is proved as follows. Proposition~\ref{proposition_bigness_algebraic_representations} implies that
 $r(\widetilde{G}(k))$ is $m$-big.  Then, Proposition~\ref{proposition_m_big_iff_scalar_m_big} and Proposition~\ref{proposition_m_big_subgroup_m_big} allow one to deduce that 
  $\overline{\rho}(\Delta)$ is $m$-big.

\end{proof}

\section{$m$-bigness for compatible systems}
\label{section_m_bigness_for_compatible_systems}

\begin{definition}
A \emph{group with Frobenii} is a triple 
$\left(\Gamma, \mathcal{P}, \{\mathcal{F}_\alpha\}_{\alpha\in P}\right)$ where
$\Gamma$ is a profinite group, $\mathcal{P}$ is an index set and
$\{\mathcal{F}_\alpha\}_{\alpha\in \mathcal{P}}$ is a dense set of elements of $\Gamma$.
The $\mathcal{F}_\alpha$ are called the \emph{Frobenii} of the group.
\end{definition}
\begin{remark}
If $F$ is a number field then the corresponding global Galois group
$\gal(\overline{\mathbf{Q}}/ F)$ is naturally a group with Frobenii.
\end{remark}

\begin{definition}
A \emph{compatible system of representations} (with \emph{coefficients} in a number field $E$)
is a triple $(L, \mathcal{X}, \{\rho_\lambda\}_{\lambda \in L})$ where $L$ is a set of places of $E$, 
$\mathcal{X} \subset \mathcal{P} \times L$ is a subset and each 
$\rho_\lambda : \Gamma \rightarrow GL_n(E_\lambda)$ is a continuous representation, such that the following conditions are satisfied.
\begin{enumerate}
\item[-] For all $\alpha \in \mathcal{P}$, the set $\left\{ \lambda \in L : (\alpha,\lambda) \not\in \mathcal{X} \right\}$ is finite.
\item[-] For all finite sets of places $\lambda_1, \ldots, \lambda_k \in L$, the set
	$\cap_{i=1}^k \left\{ \mathcal{F}_\alpha : (\alpha, \lambda_i) \in \mathcal{X} \right\}$ is dense in $\Gamma$.
\item[-] For all $(\alpha, \lambda) \in \mathcal{X}$, the characteristic polynomial of
		$\rho_\lambda(\mathcal{F}_\alpha)$ has coefficients in $E$ and depends only upon $\alpha$.
\end{enumerate}
%The compatible system is said to be \emph{semi-simple} if each representation $\rho_\lambda$ is semi-simple.  
The set of places $L$ is said to be \emph{full} if there exists a set $L'$ of rational primes of Dirichlet density $1$ such that for all places $\lambda$ of $E$ lying above an $l \in L'$, we have that $\lambda \in L$.
\end{definition}

The main theorem can now be stated.  It is a mild generalization of~\cite[Theorem 8.1]{snowden_wiles} and is proved in the same way
by simply appealing to Proposition~\ref{proposition_bigness_nearly_hyperspecial_groups_improvement_wiles_snowden} instead of~\cite[Proposition 6.1]{snowden_wiles} 
\begin{theorem}
Let $m$ be a positive integer,
 let $\left(\Gamma, \mathcal{P}, \{\mathcal{F}_\alpha\}_{\alpha\in \mathcal{P}}\right)$ be a group with Frobenii,
 let $E$ be a Galois extension of $\mathbf{Q}$, let $L$ be a full set of places of $E$ and for each
 $w \in L$, let $\rho_w : \Gamma \rightarrow GL_n(E_w)$ be a continuous representation and let $\Delta_w \subset  \Gamma$ be a normal open subgroup.  Assume the following properties are satisfied.
 \begin{enumerate}
 	\item[-] The $\rho_w$ form a compatible system of representations. 
 	\item[-] $\rho_w$ is absolutely irreducible when restricted to any open subgroup of 
	$\Gamma$ for all $w \in L$.
	\item[-] $\Gamma/ \Delta_w$ is cyclic of order prime to $l$ 
		where $l$ denotes the residual characteristic of $w$.
	\item[-]
	$	 [\Gamma : \Delta_w]  \rightarrow \infty $ as $w \rightarrow \infty$.
 \end{enumerate}
 Then there exists a set of places $P$ of $\mathbf{Q}$ of Dirichlet density $1/[E:\mathbf{Q}]$, all of which split completely in $E$, such that, for all $w \in L$ lying above a place  $l \in P$:
 \begin{enumerate}
 \item[i)] $\overline{\rho}_w(\Delta_w)$ is an $m$-big subgroup of  $GL_n(\mathbf{F}_l)$ .
 \item[ii)] $[\ker \ad\ \overline{\rho}_w : \Delta_w \cap \ker \ad\ \overline{\rho}_w] > m$.
 \end{enumerate}
 \end{theorem}

Theorem~\ref{theorem_introduction_theorem} is then the special case of the above theorem where $\Gamma$ is the absolute Galois group of a number field.

%la version abrig
\section{version fran{\c c}aise abr{\'e}g{\'e}e}
\selectlanguage{francais}
Le but de cet article est de faire les modifications n\'ecessaires au travail de Snowden-Wiles~\cite{snowden_wiles}
afin de g\'en\'eraliser leurs r\'esultats sur $1$-\emph{big} \`a \emph{$m$-big}.
La d\'efinition de \emph{$m$-big} est rappel\'ee dans Definition~\ref{definition_mbig}.
Elle est une condition technique qui appara\^it dans les g\'en\'eralisations r\'ecentes de la m\'ethode de Taylor-Wiles 
aux groupes unitaires (cf.~\cite{taylor_et_al_modularity_results}).
Le r\'esultat principal de cet article est le th\'eor\`eme suivant (cf. Theorem~\ref{theorem_introduction_theorem}).
\begin{theoreme}
Soient $m \in \mathbf{N}$, $F$ un corps de nombres, $E$ une extension galoisienne de $\mathbf{Q}$ et $L$ un ensemble plein de places de $E$.
Pour tous $w \in L$, soient  $\rho_w : \gal(\overline{\mathbf{Q}}/F) \rightarrow GL_n(E_w)$ une repr\'esentation continue semi-simple 
et $\Delta_w \subset \gal(\overline{\mathbf{Q}}/F)$ un sous-groupe normal ouvert.
Supposons que les propri\'et\'es suivantes sont satisfaites:
 \begin{enumerate}
 	\item[-] Les $\rho_w$ forment un syst\`eme compatible de repr\'esentations. 
 	\item[-] Pour tout $w \in L$, la restriction de $\rho_w$ \`a n'importe quel sous-groupe ouvert de
	de 
	$\gal(\overline{\mathbf{Q}}/F)$ est absolument irr\'eductible.
	\item[-] 
	Pour tout $w \in L$,
	$\gal(\overline{\mathbf{Q}}/F) / \Delta_w$ est cyclique d'ordre premier \`a 
		 la caract\'eristique r\'esiduelle de $w$.
	\item[-]
	$	 [\gal(\overline{\mathbf{Q}}/F) : \Delta_w]  \rightarrow \infty$
		lorsque $w \rightarrow \infty$.
 \end{enumerate}
Alors il existe un ensemble de places $P$ de $\mathbf{Q}$ de densit\'e  $1/[E:\mathbf{Q}]$, 
qui sont toutes totalement 
d\'eploy\'ees dans $E$, et telles que, pour tout $w \in L$ au-dessus une place $l \in P$:
 \begin{enumerate}
 \item[i)] $\overline{\rho}_w(\Delta_w)$ est un sous-groupe $m$-big de  $GL_n(\mathbf{F}_l)$ .
 \item[ii)] $[ \ker \ad \overline{\rho}_w : \Delta_w \cap \ker \ad \overline{\rho}_w] > m$.
 \end{enumerate}
\end{theoreme}
\begin{remarque}
La premi\`ere partie de ce th\'eor\`eme est une g\'en\'eralisation  du r\'esultat principal de Snowden-Wiles~\cite{snowden_wiles}.  La d\'emonstration suit leurs arguments, en appliquant Proposition~\ref{proposition_stronger_version_higher_regular_elements_sw} au lieu de~\cite[Proposition 4.1]{snowden_wiles}.

La deuxi\`eme partie est un r\'esultat de Barnet-Lamb-Gee-Geraghty-Taylor qui est apparu \`a l'origine dans~\cite{taylor_et_al_modularity_results}.
\end{remarque}

Le plan de cet article est pareil \`a celui de~\cite{snowden_wiles}.
Section~\ref{section_elementary_properties_m_bigness} d\'emontre quelques propri\'et\'es de $m$-big qui \'etaient 
d\'emontr\'ees pour $1$-big dans~\cite[\S2]{snowden_wiles}.  Section~\ref{section_highly_regular_elements_semi_simple_groups} d\'emontre Proposition~\ref{proposition_stronger_version_higher_regular_elements_sw} qui am\'eliore l\'eg\`erement~\cite[Proposition 4.1]{snowden_wiles}.  Ce r\'esultat est appliqu\'e dans Section~\ref{section_m_big_algebraic_representations} 
pour d\'emontrer que l'image de certaines repr\'esentations alg\'ebriques est $m$-big (cf.
 Proposition~\ref{proposition_bigness_algebraic_representations} qui 
am\'eliore~\cite[Proposition 5.1]{snowden_wiles}).  Finalement, Section~\ref{section_mbig_nearly_hyperspecial_groups}
et Section~\ref{section_m_bigness_for_compatible_systems} appliquent ce r\'esultat pour d\'emontrer le th\'eor\`eme principal.


\begin{thebibliography}{00}

\bibitem{taylor_et_al_modularity_results} T. Barnet-Lamb, T. Gee, D. Geraghty and R. Taylor, 
	\emph{Potential automorphy and change of weight}, preprint, 2010.

\bibitem{BGHT} T. Barnet-Lamb, D. Geraghty, M. Harris and R. Taylor,
	\emph{A family of Calabi-Yau varieties and potential automorphy {II}},
	preprint, 2010.
	
\bibitem{snowden_wiles} A. Snowden and A. Wiles,
	\emph{Bigness in compatible systems}, preprint, 2009
	
\bibitem{steinberg_lectures} R. Steinberg,
	\emph{Lectures on Chevalley groups},
	Yale University, New Haven, Conn., 1968,
	Notes prepared by J. Faulkner and R. Wilson.

\end{thebibliography}
\end{document}